\documentclass[12pt]{amsart}
\usepackage{amsthm}
\usepackage{amsmath}
\usepackage{amssymb}
\usepackage{mathrsfs}
\usepackage{enumerate}
\usepackage{float}
\usepackage{pgf}
\usepackage{url, multicol}
\usepackage[all]{xy}
\usepackage{graphicx}
\usepackage{hyperref}
\usepackage[noabbrev,nosort,capitalise]{cleveref}
\usepackage[margin = 1 in]{geometry}

\newcommand{\Stir}[3]{\genfrac{[}{]}{0pt}{}{#1}{#2}_{#3}}
\renewcommand{\S}{{Z}}
\newcommand{\s}{{z}}
\renewcommand{\Re}{\operatorname{Re}}

\newtheorem{thm}{Theorem}
\crefname{thm}{Theorem}{Theorems}

\newtheorem{prop}[thm]{Proposition}
\newtheorem{lemma}[thm]{Lemma}
\newtheorem{cor}[thm]{Corollary}
\theoremstyle{definition}
\newtheorem{rem}{Remark}

\newtheorem{exam}{Example}
\numberwithin{equation}{section}
\numberwithin{table}{thm}

\allowdisplaybreaks

\title{On Product Formulas of Guillera and Sondow}
\author[Shihan Kanungo]{Shihan Kanungo}
\address{Henry M. Gunn High School \ Palo Alto, CA 94306}
\email{shihankanungo@gmail.com}
\author{Jordan Schettler}
\address{San Jos\'e State University \ San Jos\'e, CA 95192}
\email{jordan.schettler@sjsu.edu}
\date{\today}   

\begin{document}

\begin{abstract}
    In this note, we evaluate a multivariable family of infinite products which generalize Guillera's infinite product for $e$, and Ser's formula (rediscovered by Sondow) for $e^\gamma$. We describe formulas for the products in terms of special values of the Hurwitz zeta function $\zeta(s,u)$ and its $s$-derivative. Additionally, we derive integral and double integral representations for the logarithms of these infinite products.
\end{abstract}
\maketitle
\section{Introduction}\label{sec:introduction}

As noted in \cite{Sond4}, Jes\'{u}s Guillera derived the following infinite product
\begin{equation}\label{Gu}
e=\left(\frac{2^1}{1^1}\right)^{1/1} \left(\frac{2^2}{1^1\cdot 3^1}\right)^{1/2}  \left(\frac{2^3 \cdot 4^1}{1^1 \cdot 3^3 }\right)^{1/3} \left(\frac{2^4 \cdot 4^4}{1^1 \cdot 3^6 \cdot 5^1}\right)^{1/4} \cdots = \prod_{n=1}^{\infty} t_n^{1/n}
\end{equation}
where the $n$th factor is the $n$th root of
\begin{equation*}
    t_n \mathrel{\mathop:}= \prod_{k=0}^n(k+1)^{(-1)^{k+1}\binom{n}{k}}.
\end{equation*}
Jonathan Sondow in \cite{Sond3} rediscovered the following formula, originally due to Ser in \cite{Ser}, which bears a striking resemblance to Guillera's formula:
\begin{equation}\label{So}
    e^{\gamma}=\left(\frac{2^1}{1^1}\right)^{1/2} \left(\frac{2^2}{1^1\cdot 3^1}\right)^{1/3}  \left(\frac{2^3 \cdot 4^1}{1^1 \cdot 3^3 }\right)^{1/4} \left(\frac{2^4 \cdot 4^4}{1^1 \cdot 3^6 \cdot 5^1}\right)^{1/5} \cdots = \prod_{n=1}^{\infty} t_n^{1/(n+1)}
\end{equation}
where
\[\gamma = \lim_{N\rightarrow \infty}\left(\sum_{n=1}^N \frac{1}{n} - \log N\right)\]
is Euler's constant and now the $n$th factor is the $(n+1)$th root of $t_n$.
In this note, we derive a general product formula which includes the following evaluations as specials cases:
\begin{equation}\label{first}
\sqrt{\frac{2\pi}{e}} = \left(\frac{2^1}{1^1}\right)^{1/3} \left(\frac{2^2}{1^1\cdot 3^1}\right)^{1/4}  \left(\frac{2^3 \cdot 4^1}{1^1 \cdot 3^3 }\right)^{1/5} \left(\frac{2^4 \cdot 4^4}{1^1 \cdot 3^6 \cdot 5^1}\right)^{1/6} \cdots,
\end{equation}
\begin{equation}\label{second}
\sqrt[4]{\frac{2\pi}{e^{3/2}}} \cdot A =\left(\frac{2^1}{1^1}\right)^{1/4} \left(\frac{2^2}{1^1\cdot 3^1}\right)^{1/5}  \left(\frac{2^3 \cdot 4^1}{1^1 \cdot 3^3 }\right)^{1/6} \left(\frac{2^4 \cdot 4^4}{1^1 \cdot 3^6 \cdot 5^1}\right)^{1/7} \cdots,
\end{equation}
where $A$ is the Glaisher-Kinkelin constant, and
\begin{equation}\label{third}
\sqrt[6]{\frac{2\pi}{e^{7/4}}} \cdot A \exp(\zeta(3)/(8\pi^2))= \left(\frac{2^1}{1^1}\right)^{1/5} \left(\frac{2^2}{1^1\cdot 3^1}\right)^{1/6}  \left(\frac{2^3 \cdot 4^1}{1^1 \cdot 3^3 }\right)^{1/7} \left(\frac{2^4 \cdot 4^4}{1^1 \cdot 3^6 \cdot 5^1}\right)^{1/8} \cdots,
\end{equation}
where $\zeta(3) = 1/1^3+1/2^2+1/3^3+\cdots$ is Ap\'{e}ry's constant.
In their joint paper \cite{Guil}, Guillera and Sondow evaluate many products related to \cref{Gu,So} (see also \cite{Hadj}).
For example, on pages 267--268 they show that for $u>0$ and each integer $j>0$
\begin{align}\label{exponent shift}
    e^{B(u,j)} =  \prod_{n=1}^{\infty} \left(\prod_{k=0}^{n+j-1}(k+u)^{(-1)^{k+1}\binom{n+j-1}{k}}\right)^{1/n}
\end{align}
where $B(u,v) = \Gamma(u)\Gamma(v)/\Gamma(u+v)$ is the Euler beta function.
In particular, when $u=1$, this recovers the following evaluation originally due to Goldschmidt and Martin \cite{Gold}:
\begin{equation}\label{Gold}
e^{1/j} = \prod_{n=1}^{\infty} t_{n+j-1}^{1/n}.
\end{equation}
The above can be thought of as a generalization of Guillera's \cref{Gu} obtained by fixing the exponents $1/n$ and shifting the bases from $t_n$ to $t_{n+j-1}$.
Here we consider shifts in the exponents from $1/n$ to $1/(n+\alpha+1)$ for a complex number $\alpha\neq -2, -3, -4, \ldots$.
We also interpolate the factors $k+1$ in $t_n$ by $k+u$ for real $u>0$ (as in \cref{exponent shift}). In other words, we define 
\begin{equation*}
    t_n(u) \mathrel{\mathop:}= \prod_{k=0}^n(k+u)^{(-1)^{k+1}\binom{n}{k}}
\hspace{2 em}\mbox{ and }\hspace{2 em}
    \s_\alpha(u) \mathrel{\mathop:}= \prod_{n=1}^{\infty} t_n(u)^{1/(n+\alpha+1)}.
\end{equation*}
We will also write $z_{\alpha}$ for $z_{\alpha}(1)$.
Guillera's \cref{Gu} is then equivalent to $\s_{-1}=e$, while
Sondow and Ser's \cref{So} is equivalent to $\s_0=e^\gamma$.
In the present work, we establish \cref{main} below for nonnegative integer shifts $\alpha=d$ on the exponents and positive real parameters $u$ in the factors $t_n(u)$.
Our evaluations are given in terms of special values of the Hurwitz zeta function $\zeta(s,u)$ and its $s$-derivative $\zeta'(s,u)$.

We define for $\alpha\neq-1,-2,-3,\ldots$
\begin{align*}
    \S_\alpha(s,u) = \sum_{n=0}^\infty\frac{1}{n+\alpha+1} \sum_{k=0}^n (-1)^k \binom{n}{k} (k+u)^{1-s}.
\end{align*}
Note that $\S_\alpha(s,u)-\frac{u^{1-s}}{\alpha+1}$ is valid at $\alpha=-1$, and
it is easily verified that $\log(\s_{\alpha}(u)) = \frac{\partial}{\partial s} (\S_\alpha(s,u)-\frac{u^{1-s}}{\alpha+1})|_{s=1}$ for $\alpha\neq -2, -3, -4, \ldots$.
The first result we state is given in terms of shifted $r$-Stirling numbers $\Stir{n+r}{m+r}{r}$ defined by the generating function
\begin{align*}
    \sum_{m=0}^n \Stir{n+r}{m+r}{r}x^m = (x+r)(x+r+1)\cdots(x+r+n-1).
\end{align*}
\begin{thm}\label{main}
   For $u>0$ and each integer $d\geq 0$, we have
    \begin{align*}
        \S_d(s,u)=
    \frac{1}{d!}\sum_{k=0}^d \Stir{d+1-u}{k+1-u}{1-u}(s-k-1)\zeta(s-k,u).
\end{align*}
Consequently,
\begin{align}\label{sum}
    \log(\s_{d}(u)) = \frac{\log(u)}{d+1}+\frac{1}{d!}\sum_{k=0}^d\Stir{d+1-u}{k+1-u}{1-u}(\zeta(1-k,u)-k\zeta'(1-k,u)).
\end{align}
\end{thm}
\begin{rem}\label{entire function}
Each $\S_{\alpha}(s,u)$ is an entire function of $s$ for fixed $\alpha \neq -1, -2, -3, \ldots$ and $u>0$; this can be seen using ideas of Hasse \cite{Hass}, also found in \cite{Guil} (proof of Theorem 2.2, stated as \cref{zeta1} below) and \cite{Sond1994} (proof of the Theorem in Section 3). 
Likewise, the function $\zeta(s,u)+(s-1)\zeta'(s,u)$ can be regarded as an entire function of $s$ with value $-\psi(u)$ at $s=1$ where $\psi$ is the digamma function since in a neighborhood around $s=1$ we have the Laurent series
\[\zeta(s,u) = \frac{1}{s-1} -\psi(u)+ c_1(s-1)+\cdots,\]
and
\[(s-1)\zeta'(s,u) = - \frac{1}{s-1} + c_1(s-1)+\cdots.\]
This ensures that the terms of the sum in \cref{sum} remain valid at $k=0$. 
In particular, setting $d=0$ in \cref{main} and using $\Stir{1-u}{1-u}{1-u}=1$, we get the evaluation $\log(\s_0(u))=\log(u)-\psi(u)$, which is equivalent to Sondow and Ser's \cref{So} when $u=1$ and Theorem 5.1 in \cite{Guil} for general $u$.
Moreover, the $d=1$ case of \cref{main} can be derived from the work of Blagouchine \cite{Blag} (see Equation (129) on page 33).
\end{rem}
\begin{rem}
Note that we can relate the shifted $r$-Stirling numbers $\Stir{d+1-u}{k+1-u}{1-u}$ in \cref{main} to the usual unsigned Stirling numbers of the first kind $\Stir{n}{m}{}=\Stir{n}{m}{0}$ via
\begin{align*}
    \Stir{d+1-u}{k+1-u}{1-u} = \sum_{\ell=k}^d \genfrac{[}{]}{0pt}{}{d}{\ell} \binom{\ell}{k}(1-u)^{\ell-k}.
\end{align*}
Also, much like the usual Stirling numbers or binomial coefficients, the shifted $r$-Stirling numbers satisfy a simple triangular recurrence relation:
\begin{equation}\label{relation}
    (n+r)\Stir{n+r}{m+1+r}{r} + \Stir{n+r}{m+r}{r} = \Stir{n+1+r}{m+1+r}{r}.
\end{equation}
See \cite{Brod} for further properties and identities of $r$-Stirling numbers.
\end{rem}
\begin{rem}
    When $m=1-s$ is a positive integer, then $-m\zeta(1-m,u)=B_m(u)$ is the $m$th Bernoulli polynomial, and since $n>m$ implies $\sum_{k=0}^n (-1)^k\binom{n}{k}(k+u)^m=0$, the first statement in \cref{main} in this case becomes
\[\sum_{n=0}^m \frac{1}{n+d+1}\sum_{k=0}^n (-1)^k\binom{n}{k}(k+u)^m =  \frac{1}{d!}\sum_{k=0}^d \Stir{d+1-u}{k+1-u}{1-u}B_{m+k}(u) \]
which reduces to the well-known explicit formula for $B_m(u)$ when $d=0$.
\end{rem}

We will also establish integral formulas for $\log (\s_{\alpha-1}(u))$ analogous to those found by Sondow for Euler's constant $\gamma$ in \cite{Sond2}.
\begin{thm}\label{integral}
    For $\Re(\alpha)>-1$ and $u>0$, we have
    \begin{align}\label{cont double integral eqn1}
        \log(\s_{\alpha-1}(u)) = \iint_{[0,1]^2} - \frac{(1-p)^\alpha(pq)^{u-1}}{(1-pq)^\alpha\log pq}\, dp\, dq.
    \end{align}
Further, when $\alpha=d\geq 0$ is an integer, we have
\begin{align*}
  \log(\s_{d-1}(u))=\int_0^1 x^{u-1}\left(\frac{1}{(1-x)^d}+\frac{1}{\log x}\sum_{m=1}^{d}\frac{1}{m(1-x)^{d-m}}\right) \, dx.
\end{align*}
\end{thm}
The paper is organized as follows.
In \cref{exsec}, we look at some consequences of \cref{main,integral} and prove a result, \cref{explicit}, which gives explicit formulas for $z_d=z_{d}(1)$.
In \cref{proof}, we will establish a functional equation for $\S_{\alpha}(s,u)$ and prove \cref{main,integral}.

\section{Examples}\label{exsec}
In this section, we exhibit several examples, illustrating applications of \cref{main,integral}.
\begin{exam}
Setting $u=1$, we get $\Stir{d+1-1}{k+1-1}{1-1} =  \Stir{d}{k}{}$, so when $d=1$, \cref{main} implies
    \[\log(\s_1) = \zeta(0)-\zeta'(0) = -\frac{1}{2}+\frac{1}{2}\log(2\pi),\]
    giving the infinite product seen in \cref{first}:
    \[ \sqrt{\frac{2\pi}{e}} = \s_1 =\left(\frac{2^1}{1^1}\right)^{1/3} \left(\frac{2^2}{1^1\cdot 3^1}\right)^{1/4}  \left(\frac{2^3 \cdot 4^1}{1^1 \cdot 3^3 }\right)^{1/5} \left(\frac{2^4 \cdot 4^4}{1^1 \cdot 3^6 \cdot 5^1}\right)^{1/6} \cdots. \]
    Note that an alternate infinite product for $\sqrt{2\pi/e}$ was given in \cite{Guil} (page 266, Example 5.6) where the exponents are the same as in Ser's and Sondow's formula for $e^\gamma$ but with different factors, whereas here we have the same factors $t_n$ but shifted exponents $1/(n+2)$ instead of $1/(n+1)$.
    More generally for $d=1$ and any $u>0$, we use the values $\Stir{2-u}{2-u}{1-u}=1$, $\Stir{2-u}{1-u}{1-u}=1-u$, $\zeta(0,u) = \frac{1}{2}-u$, and $\zeta'(0,u)=\log\Gamma(u)-\frac{1}{2}\log(2\pi)$ (\cite{Lerc}) to get
    \[ \log(\s_1(u)) = \frac{\log(u)}{2}+(u-1)\psi(u) +\frac{1}{2}-u-\log\Gamma(u) +\frac{1}{2}\log (2\pi).\]
    In particular, using $\psi(1/2)=-2\log 2-\gamma$ and $\Gamma(1/2)=\sqrt{\pi}$ yields
    \[2\sqrt{e^\gamma}=z_1(1/2)=\left(\frac{3^1}{1^1}\right)^{1/3} \left(\frac{3^2}{1^1 \cdot 5^1}\right)^{1/4}  \left(\frac{3^3 \cdot 7^1}{1^1 \cdot 5^3 }\right)^{1/5}
    \left(\frac{3^4\cdot 7^4}{1^1\cdot 5^6 \cdot 9^1}\right)^{1/6}
    \cdots ,\]
    and similarly using $\psi(1/3) = -\frac{\pi}{2\sqrt{3}}-\frac{3}{2}\log 3 - \gamma$ and $\Gamma(1/3) = \frac{2^{7/9}\pi^{2/3}}{3^{1/12}\mathrm{AGM}(2,\sqrt{2+\sqrt{3}})^{1/3}}$ gives
    \begin{align*}
    &\sqrt[6]{\frac{3^{7/2}}{2^{5/3}\cdot \pi}}\cdot \left(\exp\left(\frac{1}{2}+\frac{\pi}{\sqrt{3}}+2\gamma\right)\mathrm{AGM}\left(2,\sqrt{2+\sqrt{3}}\right)\right)^{1/3} = z_1(1/3)\\
    =& \left(\frac{4^1}{1^1}\right)^{1/3} \left(\frac{4^2}{1^1 \cdot 7^1}\right)^{1/4}  \left(\frac{4^3 \cdot 10^1}{1^1 \cdot 7^3 }\right)^{1/5}
    \left(\frac{4^4\cdot 10^4}{1^1\cdot 7^6 \cdot 13^1}\right)^{1/6}
    \cdots
    \end{align*}
    where AGM denotes an arithmetic-geometric mean.
    We can also cancel the contribution from the digamma function $\psi(u)$ by considering the combination
    \[(u-1)\log(\s_0(u)) + \log(\s_1(u)) = \left(u-\frac{1}{2}\right)\log(u)+\frac{1}{2}-u-\log\Gamma(u)+\frac{1}{2}\log(2\pi)\]
    which leads to the following formula when $u=2$:
    \begin{align*}
        4\sqrt{\frac{\pi}{e^3}} = \s_0(2)\s_1(2) = \left(\frac{3^1}{2^1}\right)^{5/6} \left(\frac{3^2}{2^1 \cdot 4^1}\right)^{7/12}  \left(\frac{3^3 \cdot 5^1}{2^1 \cdot 4^3 }\right)^{9/20}
    \left(\frac{3^4\cdot 5^4}{2^1\cdot 4^6 \cdot 6^1}\right)^{11/30}
    \cdots 
    \end{align*}
    where the exponents are of the form $(2n+3)/((n+1)(n+2))$.
    \end{exam}
    \begin{exam}
    When $d=2$ and $u=1$ \cref{main} implies
     \[\log(\s_2) = \frac{1}{2}(\zeta(0)-\zeta'(0) +\zeta(-1)-2\zeta'(-1))= -\frac{3}{8} + \frac{1}{4}\log(2\pi) + \log A\]
    where $A$ is the Glaisher-Kinkelin constant, which gives the product in \cref{second}
    \[ \sqrt[4]{\frac{2\pi}{e^{3/2}}} \cdot A = \s_2 =\left(\frac{2^1}{1^1}\right)^{1/4} \left(\frac{2^2}{1^1\cdot 3^1}\right)^{1/5}  \left(\frac{2^3 \cdot 4^1}{1^1 \cdot 3^3 }\right)^{1/6} \left(\frac{2^4 \cdot 4^4}{1^1 \cdot 3^6 \cdot 5^1}\right)^{1/7} \cdots. \]
    The Glaisher-Kinkelin constant $A$ and the hyperfactorial are connected in the same way that $\sqrt{2\pi}$ and the factorial are connected by Stirling's asymptotic. In particular, we have
    \[1\cdot 2\cdot 3 \cdots n \sim \sqrt{2\pi} \frac{n^{n+1/2}}{e^{n}} \]
    and
    \[1^1 \cdot 2^2\cdot  3^3 \cdots n^n \sim A \frac{n^{(1/2)n^2+(1/2)n+1/12}}{e^{n^2/4}}\]
as $n\rightarrow \infty$.
Bendersky \cite{Bend} generalized these results showing that for each integer $k\geq 0$, there exists a constant $A_k$ and two degree $k+1$ polynomials $P_k$, $Q_k$ with rational coefficients such that
\[1^{1^k}\cdot 2^{2^k} \cdot 3^{3^k} \cdots n^{n^k}\sim A_k\frac{n^{P_k(n)}}{e^{Q_k(n)}} \]
as $n\rightarrow \infty$.
In particular, $A_0=\sqrt{2\pi}$ and $A_1=A$.
Adamchik \cite{Ada98} later rediscovered these asymptotics and also described the logarithms of the constants in terms of the Riemann zeta function and its derivative:
\begin{equation}\label{adamchik1}
    \log A_k = (-1)^kH_k\zeta(-k)-\zeta'(-k)
\end{equation}
for each integer $k\geq 0$ where $H_k = 1/1+1/2+\cdots1/k$ is the $k$th harmonic number with $H_0=0$.
Using Adamchik's formula along with the fact that
\begin{align}\label{valueofzeta}
    \zeta(-k+1) = (-1)^{k-1}\frac{B_k}{k},
\end{align}
for $k\geq 1$ where $B_k$ denotes the $k$th Bernoulli number\footnote{Recall that $B_k$ is a rational number given by the generating function $\frac{x}{e^x-1} = \sum_{k=0}^{\infty}B_k\frac{x^k}{k!}$ with $B_k=0$ for odd $k>1$.},
we obtain the following explicit version of the second statement in \cref{main} when $u=1$.
\end{exam}
\begin{cor}\label{explicit}
    For integers $d\geq 1$
    \begin{align*}
        \log(\s_d)=
        \frac{1}{2d}(-1+\log(2\pi)) +\frac{1}{d!}\left(- \sum_{k=1}^{\lfloor d/2\rfloor}\genfrac[]{0pt}{0}{d}{2k}H_{2k}B_{2k} + \sum_{k=1}^{d-1} \genfrac[]{0pt}{0}{d}{k+1}(k+1)\log(A_{k})\right).
    \end{align*}
\end{cor}
\begin{proof}
First note that for $d\ge 1$, we have $\genfrac{[}{]}{0pt}{}{d}{0}=0$ so we can start the summation at index $1$.
Now rewriting \cref{adamchik1}, we have
\begin{align}\label{adamchik2}
    \zeta'(-k+1) = \frac{H_{k-1}B_{k}}{k}- \log(A_{k-1}).
\end{align}
Using \cref{valueofzeta,adamchik2}, we get 
\begin{align*}
    \log(\s_{d}) &= \frac{1}{d!}\sum_{k=1}^d \genfrac[]{0pt}{0}{d}{k}\left((-1)^{k-1} \frac{B_k}{k} - H_{k-1}B_k + k\log(A_{k-1})\right)\\
    &= \frac{1}{d!}\left((d-1)!\left(\log(A_0)-\frac{1}{2}\right) - \sum_{k=1}^{\lfloor d/2\rfloor}\genfrac[]{0pt}{0}{d}{2k} H_{2k}B_{2k} + \sum_{k=1}^{d-1}\genfrac[]{0pt}{0}{d}{k+1} (k+1)\log(A_{k})\right).
\end{align*}
Since $\log(A_0) = \frac{1}{2}\log(2\pi)$, we obtain \cref{explicit}.
\end{proof}
\begin{exam}
Setting $d=3$ in \cref{explicit}, we get
\begin{align*}
 \log(\s_3)&=\frac{1}{6}(-1+\log(2\pi))+\frac{1}{6}\left(-3\cdot \frac{3}{2}\cdot \frac{1}{6} + 3(1+1)\log(A_1) + 1\cdot (2+1)\log(A_2)\right) \\
 &= -\frac{7}{24} +\frac{1}{6}\log(2\pi) + \log(A) + \frac{1}{2}\log(A_2)
\end{align*}
Differentiating the functional equation of the Riemann zeta function and using Euler's formula for the values of $\zeta(s)$ at positive even integers we have
\[\log(A_{2k}) =(-1)^{k+1} \frac{(2k)!}{2 (2\pi)^{2k}} \cdot \zeta(2k+1). \]
In particular, we have the expression $\log(A_2) = \zeta(3)/(2\pi)^2$ involving the famously known irrational $\zeta(3)$ (Ap\'{e}ry's constant), so we get the product in \cref{third}
\[\sqrt[6]{\frac{2\pi}{e^{7/4}}} \cdot A \exp(\zeta(3)/(8\pi^2))=\s_3= \left(\frac{2^1}{1^1}\right)^{1/5} \left(\frac{2^2}{1^1\cdot 3^1}\right)^{1/6}  \left(\frac{2^3 \cdot 4^1}{1^1 \cdot 3^3 }\right)^{1/7} \left(\frac{2^4 \cdot 4^4}{1^1 \cdot 3^6 \cdot 5^1}\right)^{1/8} \cdots.\]
\end{exam}

\begin{exam}
Using $\alpha=d=0$ in \cref{integral} gives
\[\log(\s_{-1}(u)) = \int_0^1 x^{u-1}\, dx =1/u,\]
so we conclude $e^{1/u} = \s_{-1}(u)$; this allows us to recover Guillera's formula in \cref{Gu}.
Likewise, we set $\alpha=d=1$ in \cref{integral} to recover the well-known integral formula for $\psi$:
\begin{align*}
    -\psi(u)&=-\log(u)+\log(\s_0(u))\\
    &= -\log(u)+\int_0^1 x^{u-1}\left(\frac{1}{1-x}+\frac{1}{\log x} \right)\, dx \\
    &=\int_0^\infty \left(-\frac{e^{-t}}{t}+ \frac{e^{-ut}}{1-e^{-t}}\right)\, dt
\end{align*} 
via the substitution $x = e^{-t}$ and the identity $$\int_0^{\infty}\frac{e^{-t}-e^{-tu}}{t}\, dt =\displaystyle\lim_{z\rightarrow 0^{+}} (\mathrm{Ei}(-zu)-\mathrm{Ei}(-z)) =\log(u)$$
which follows from a Puiseux series expansion for the exponential integral
\[\mathrm{Ei}(-z)=\gamma +\log z+\sum _{k=1}^{\infty }\frac {(-z)^{k}}{k\cdot k!}.\]
For $\alpha=d=2, u>0$, \cref{integral} gives the evaluations
\begin{align*}
\iint_{[0,1]^2} - \frac{(1-p)^2(pq)^{u-1}}{(1-pq)^2 \log pq}\, dp\, dq =&\int_0^1 x^{u-1}\left(\frac{1}{(1-x)^2}+\frac{1}{\log x}\left(\frac{1}{1-x}+\frac{1}{2}\right) \right)\, dx \\
= &(u-1)\psi(u) +\frac{1}{2}-u+\log\left(\frac{\sqrt{2\pi u}}{\Gamma(u)}\right).
\end{align*}
Similarly, for $\alpha=d=3, u=1$, we get
\begin{align*}
\iint_{[0,1]^2} - \frac{(1-p)^3}{(1-pq)^3 \log pq}\, dp\, dq =
    &\int_0^1 \left(\frac{1}{(1-x)^3}+ \frac{1}{\log x}\left(\frac{1}{(1-x)^2}+\frac{1}{2(1-x)}+\frac{1}{3}\right)\right)\, dx \\
    = &-\frac{3}{8}+ \frac{1}{4}\log(2\pi)+\log(A).
\end{align*}
\end{exam}
\begin{exam}
    When $\alpha$ is not an integer, finding closed forms in terms of known constants may not always be possible, but we can sometimes equate the product to an integral of an elementary function via \cref{eqn: preliminary integral} found in the proof of \cref{integral}.
    For instance, taking $\alpha=1/2$ and $u=1$ yields
    \begin{align*}
        \exp\left(\int_0^1\frac{1}{\log x}\left(1-\frac{\tanh^{-1}\sqrt{1-x}}{\sqrt{1-x}}\right)\, dx\right)
        = \left(\frac{2^1}{1^1}\right)^{1/3} \left(\frac{2^2}{1^1\cdot 3^1}\right)^{1/5}  \left(\frac{2^3 \cdot 4^1}{1^1 \cdot 3^3 }\right)^{1/7}
        \cdots .
    \end{align*}
\end{exam}

\section{Proofs}\label{proof}
In this section we prove \cref{main,integral}.
\subsection{Proof of Main Theorem}\label{mainproof}
In this subsection, we prove \cref{main}. We start with some lemmas and a functional equation for $\S_d(s,u)$ and then prove the main theorem. The first lemma is Hasse's formula for the Hurwitz zeta function.
\begin{lemma}[Hasse,\cite{Hass}]\label{zeta1}
    For all complex $s$,
    \begin{equation}
        (s-1)\zeta(s,u) = \sum_{n=0}^{\infty}\frac{1}{n+1} \sum_{k=0}^n (-1)^k\binom{n}{k}(k+u)^{1-s}.
    \end{equation}
    where $\zeta(s,u)$ is the Hurwitz zeta function.
\end{lemma}
Note that $(s-1)\zeta(s,u)$ can be regarded as an entire function with value $1$ at $s=1$, which allows us to state \cref{zeta1} for all $s$ rather than $s\ne 1$.
\begin{lemma}\label{log2}
    For $u>0$ and any $s$, we have
    \begin{align*}
        \sum_{n=0}^\infty \sum_{k=0}^n (-1)^{k}\binom{n}{k}(k+u+1)^{1-s} = u^{1-s}
    \end{align*}
\end{lemma}
\begin{proof}
Fix $u>0$ and set $v=u+1$.
Using the method of Hasse \cite{Hass} we can show the double sum converges for all $s$ and represents an entire function.
By analytic continuation, it suffices to prove equality in the half plane $\Re(s)>1$.
To do so, we use the identity
\[\frac{1}{\Gamma(s-1)}\int_0^\infty e^{-wt}t^{s-2} \, dt = w^{1-s}\]
to transform the double sum:
\begin{align*}
    \sum_{n=0}^\infty \sum_{k=0}^n (-1)^{k}\binom{n}{k}(k+v)^{1-s} &= \frac{1}{\Gamma(s-1)} \sum_{n=0}^\infty \int_0^\infty \sum_{k=0}^n (-1)^{k}\binom{n}{k} e^{-(v+k)t}t^{s-2}\, dt \\
    &= \frac{1}{\Gamma(s-1)} \sum_{n=0}^\infty \int_0^\infty (1-e^{-t})^n e^{-vt}t^{s-2}\, dt \\
    &= \frac{1}{\Gamma(s-1)} \int_0^\infty e^{-(v-1)t}t^{s-2}\, dt \\
    &= (v-1)^{1-s} = u^{1-s}. \qedhere
\end{align*}
\end{proof}
\begin{rem}
     One can also give a quick formal proof of \cref{log2} by writing the left hand side of the statement in terms of the forward difference operator $\Delta[f(v)]=f(v+1)-f(v)$:
    \begin{align*}
        \sum_{n=0}^\infty (-1)^n\Delta^n [v^{1-s}] &= \left(1-\Delta+\Delta^2 - \cdots\right)[v^{1-s}]\\
        &= \left(1+\Delta\right)^{-1}[v^{1-s}].
    \end{align*}
    Now note that $(1+\Delta)[f(v)] = f(v+1)$, so $\left(1+\Delta\right)^{-1}[v^{1-s}] = (v-1)^{1-s} = u^{1-s}$.
\end{rem}
Next, we prove a functional equation for $\S_\alpha(s,u)$.
\begin{prop}\label{fe}
    Whenever $\alpha\neq 0, -1, -2, \ldots$, we have
    \begin{align*}
        \alpha \S_{\alpha}(s,u) = \S_{\alpha-1}(s-1,u)+(\alpha-u)\S_{\alpha-1}(s,u).
    \end{align*}
\end{prop}
\begin{proof}
    The following identity holds for nonnegative integers $k,n$ and any $\alpha$:
\begin{align*}
   \alpha\binom{n}{k}= (k+\alpha)\binom{n+1}{k} -(n+\alpha+1)\binom{n}{k-1}.
\end{align*}
Summing this and noting that the identity at $k=n+1$ is $0=0$ gives us
\begin{align*}
    \alpha\S_\alpha(s,u) =&\sum_{n=0}^\infty \frac{1}{n+\alpha+1} \sum_{k=0}^{n+1} (-1)^k\alpha\binom{n}{k}(k+u)^{1-s}\\
    =&\sum_{n=0}^\infty \frac{1}{n+\alpha+1} \sum_{k=0}^{n+1} (-1)^k\binom{n+1}{k}(k+u)^{1-s}(k+u+(\alpha-u))\\
    &-\sum_{n=0}^\infty \sum_{k=0}^{n+1} (-1)^k \binom{n}{k-1}(k+u)^{1-s}\\
     =&\sum_{n=1}^\infty \frac{1}{n+(\alpha-1)+1} \sum_{k=0}^{n} (-1)^k\binom{n}{k}(k+u)^{1-(s-1)}\\
     &+(\alpha-u)\sum_{n=1}^\infty \frac{1}{n+(\alpha-1)+1} \sum_{k=0}^{n} (-1)^k\binom{n}{k}(k+u)^{1-s} \\
    &+\sum_{n=0}^\infty \sum_{k=0}^{n} (-1)^k \binom{n}{k}(k+u+1)^{1-s}\\
    =& \S_{\alpha-1}(s-1,u) - \frac{u^{2-s}}{\alpha}+(\alpha-u)\S_{\alpha-1}(s,u) - (\alpha-u)\frac{u^{1-s}}{\alpha} +  u^{1-s} \\
    =&\S_{\alpha-1}(s-1,u)+(\alpha-u)\S_{\alpha-1}(s,u)
\end{align*}
where we used \cref{log2} in the next to last step.
\end{proof}

Now we can prove the main theorem.
\begin{proof}[Proof of \cref{main}]
    We induct on $d$; the base case $d=0$ is \cref{zeta1}. Now suppose $d\geq 1$ and the statement is true for $d-1$.
    Using the functional equation in \cref{fe} and the triangular recurrence relation for shifted $r$-Stirling numbers in \cref{relation}, we get
\begin{align*}
    d\S_{d}(s,u) =&\S_{d-1}(s-1,u) + (d-u)\S_{d-1}(s,u)\\ = &\frac{1}{(d-1)!}\sum_{k=0}^{d-1} \Stir{d-u}{k+1-u}{1-u} (s-(k+1)-1)\zeta (s-(k+1),u) \\
    &+ \frac{1}{(d-1)!}\sum_{k=0}^{d-1}(d-u)\Stir{d-u}{k+1-u}{1-u}(s-k-1,u)\zeta (s- k,u)\\
    =& \frac{1}{(d-1)!}\sum_{k=0}^d \left(\Stir{d-u}{k-u}{1-u}+(d-u)\Stir{d-u}{k+1-u}{1-u}\right)(s-k-1)\zeta(s-k,u)\\
    =& \frac{1}{(d-1)!}\sum_{k=0}^d \Stir{d+1-u}{k+1-u}{1-u}(s-k-1)\zeta(s-k,u).
\end{align*}
This completes the inductive step and proves the theorem.
\end{proof}

\subsection{Proof of Integral Formulas}\label{integralsec}
In this subsection, we prove \cref{integral}.

\begin{proof}[Proof of \cref{integral}]
    We will roughly follow the method outlined in \cite{Sond3} used for the corresponding integral representations of $\gamma = \log(\s_0)$.
    Start with the Euler beta integral identity below where $t>0$ and $n$ is a positive integer:
    \begin{align*}
        \int_0^1 x^{t-1}(1-x)^n\, dx = \frac{\Gamma(t)\Gamma(n+1)}{\Gamma(t+n+1)} = \frac{n!}{t(t+1)(t+2)\cdots(t+n)}.
    \end{align*}
    Next, divide both sides by $n+\alpha$, sum from $n= 1$ to $\infty$, and then integrate with respect to $t$ from $u$ to $\infty$; in other words,
    \begin{align*}
        \int_u^\infty\sum_{n=1}^{\infty}\int_0^1 \frac{x^{t-1}(1-x)^n}{n+\alpha} \, dx \, dt =  \int_u^\infty\sum_{n=1}^{\infty} \frac{n!}{(n+\alpha)t(t+1)\cdots(t+n)} \, dt.
    \end{align*}
    We may integrate the right-hand side term-by-term and use partial fractions to get
    \[\sum_{n=1}^{\infty} \frac{1}{n+\alpha} \sum_{k=0}^n (-1)^{k+1}\binom{n}{k}\log(k+u) = \log(\s_{\alpha-1}(u)).\]
    We may also interchange summation on $n$ and integration in $x$ on the left-hand side by the dominated convergence theorem since
    \begin{align*}\left|\sum_{n=1}^{\infty}\frac{x^{t-1}(1-x)^n}{n+\alpha}\right| \leq x^{t-1}\sum_{n=1}^{\infty} \left|\frac{n}{n+\alpha}\right|\cdot \frac{(1-x)^{n}}{n}\leq -c_\alpha x^{t-1}\log(x)
    \end{align*}
 where $c_\alpha = \max_{n\geq 1} \left|\frac{n}{n+\alpha}\right|$ is a positive constant depending only on $\alpha$.
Moreover,
\begin{align*}
    \int_u^\infty \int_0^1  -c_\alpha x^{t-1}\log(x)\, dx \,dt = \frac{c_{\alpha}}{u} < \infty,
\end{align*}
    so Fubini's theorem permits the interchange of integration in $x$ and $t$.
    Thus
    \begin{align}\label{eqn: preliminary integral}
    \notag    \log(\s_{\alpha-1}(u)) &= \int_0^1 \int_u^\infty \frac{x^{t-1}}{(1-x)^\alpha}\sum_{n=1}^\infty \frac{(1-x)^{n+\alpha}}{n+\alpha}\, dt\, dx \\
        &= \int_0^1 -\frac{x^{u-1}}{(1-x)^\alpha \log x}\sum_{n=1}^\infty \frac{(1-x)^{n+\alpha}}{n+\alpha}\, dx
    \end{align}
    Now for $\Re(\alpha)>-1$, we have an integral representation
     \begin{align*}
        \sum_{n=1}^\infty \frac{(1-x)^{n+\alpha}}{n+\alpha} &= \int_0^{1-x} \frac{y^\alpha}{1-y} \, dy,
    \end{align*}
  and substituting this into \cref{eqn: preliminary integral} gives
    \begin{align*}
        \log(\s_{\alpha-1}) &= \int_0^1 -\frac{x^{u-1}}{(1-x)^\alpha \log x} \left(\int_0^{1-x} \frac{y^\alpha}{1-y}\, dy\right)\, dx\\
         &= \int_0^1 \int_0^{1-x} - \frac{y^\alpha x^{u-1}}{(1-y)(1-x)^\alpha \log x}\, dy\, dx.
    \end{align*}
    The change of variables $y = 1-p$, $x = pq$ transforms the above integral into \cref{cont double integral eqn1}.

Now suppose $\alpha = d\geq 0$ is an integer.
From \cref{eqn: preliminary integral} above, we have
\begin{align*}
    \log(\s_{d-1}(u))&=\int_0^1 -\frac{x^{u-1}}{(1-x)^d \log x}\sum_{n=1}^\infty \frac{(1-x)^{n+d}}{n+d}\, dx \\
    &= \int_0^1 -\frac{x^{u-1}}{(1-x)^d \log x}\left(-\log x - \sum_{m=1}^n \frac{(1-x)^m}{m}\right)\, dx \\
\end{align*}
as needed.
\end{proof}
\section*{Acknowledgements} We would like to thank Wadim Zudilin, Jeffery Lagarias, and M.~Ram Murty for their encouragement, comments and suggestions.

\bibliographystyle{amsalpha}
\bibliography{references}
\appendix


\end{document}